%% file: SME_driven_by_CWP.tex
\documentclass[a4paper,12pt]{article}
\input{Preambule.tex}

\title{Stochastic Modified Flows, Mean-Field Limits and Dynamics of Stochastic Gradient Descent}
\titlerunning{Stochastic Modified Flows and SGD}
\author{Benjamin Gess\footnote{Fakult\"{a}t f\"{u}r Mathematik, Universit\"{a}t Bielefeld, 33615 Bielefeld, Germany.}\, \footnote{Max Planck Institute for Mathematics in the Sciences, 04103 Leipzig, Germany}\, , Sebastian Kassing$^*$, Vitalii Konarovskyi$^*$\footnote{Institute of Mathematics of NAS of Ukraine, 01024 Kyiv, Ukraine}}

\date{\today}

\def\subjclass#1{{\renewcommand{\thefootnote}{}%
\footnote{\emph{Mathematics Subject Classification (2020):} #1}}}
\def\emails#1{{\renewcommand{\thefootnote}{}%
\footnote{\emph{E-mails:} #1}}}

\begin{document}

\maketitle

\emails{\href{mailto:benjamin.gess@math.uni-bielefeld.de}{benjamin.gess@math.uni-bielefeld.de}, \href{mailto:skassing@math.uni-bielefeld.de}{skassing@math.uni-bielefeld.de},\\ \href{mailto:vitalii.konarovskyi@math.uni-bielefeld.de}{vitalii.konarovskyi@math.uni-bielefeld.de}}

\begin{abstract}
We propose new limiting dynamics for stochastic gradient descent in the small learning rate regime called stochastic modified flows. These SDEs are driven by a cylindrical Brownian motion and improve the so-called stochastic modified equations by having regular diffusion coefficients and by matching the multi-point statistics. As a second contribution, we introduce distribution dependent stochastic modified flows which we prove to describe the fluctuating limiting dynamics of stochastic gradient descent in the small learning rate - infinite width scaling regime. 
\end{abstract}

\section{Introduction}

Stochastic gradient descent algorithms (SGD), going back to~\cite{[RM51]}, are the most common way to train neural networks. Due to the non-convexity and non-smoothness of the corresponding loss landscapes, the analysis of the optimization dynamics is highly challenging. The analysis of the implicit, algorithmic bias of SGD in overparameterized networks is one of the key open problems in the understanding of the empirically observed good generalization properties of networks trained by SGD. Since the dynamics of SGD depend on many choices, like the choice of the loss function, the architecture of the network and the training data, their systematic understanding relies on the identification of universal structures that are invariant to these many degrees of freedoms, while retaining the essential properties of SGD. In recent years, several of such scaling limits and corresponding limiting dynamics have been identified. Among these, solutions to SDEs have been obtained as universal continuum objects in the small learning rate regime~\cite{E_Ma:2020,Li_Tai:2019}, while (stochastic) Wasserstein gradient flows have been found in infinite width overparameterized limits~\cite{Ch.Ba2018,Chizat_Bach:2020,Nitanda_Suzuki:2017,Mei_Montanari:2018,Rotskoff:2022,Gess_SMFE:2022,Sirignano_LLN_2020,Sirignano:2020,Javanmard_Mondelli:2020,Nguyen:2019,Rotskoff_Jelassi:2019}. In the present work, we introduce a new form of stochastic limiting dynamics which solves simultaneously three challenges met in previous works: (1) The irregularity of diffusion coefficients, (2) matching multi-point statistics, and (3) incorporating overparameterized limits. 

Before we comment on each of these aspects in a few more details, let us recall the principle setup of SGD in supervised learning. For a given training data set $\Theta \subseteq \R^{n_0}$ sampled from a probability distribution $\m$, one aims to minimize the empirical risk 
\[
R(z):=\E_{\m}\tilde R(z,\theta),\quad z \in \R^d,
\]
where $\tilde R:\R^d\times\Theta\to\R$ is a loss function. Let $\theta_n$, $n\in\N_0(:=\N\cup\{0\})$, be i.i.d. samples of training data drawn from $\m$. Then, the SGD dynamics is given by
\begin{equation}\label{equ_SGD_non_measure_depended_case}
Z_{n+1}^\eta(x)=Z_n^\eta(x)-\eta \nabla \tilde R(Z_n^\eta(x),\theta_n),\quad n \in \N_0,
\end{equation}
where $Z_0(x)=x$, $x\in\R^d$ and $\eta >0$. In particular, $Z^\eta_n$, $n\in \N_0$, allows to analyze the training dynamics of different initializations $x$ subject to the same choice of training data.

We next address the above mentioned challenges in a few more details.

\textbf{(1) The irregularity of diffusion coefficients:} In the regime of small learning rate, the foundational works of Li, Tai and E~\cite{Li:2017, Li_Tai:2019} have suggested stochastic modified equations (SME) as universal continuum limits that capture both the average gradient descent performed by SGD and its fluctuations. More precisely, it is shown that the SGD dynamics $Z^{\eta}_n$, $n \in \N_0$, with learning rate $\eta$ can be approximated to higher order in $\eta$ by solutions to SMEs
\begin{align} \begin{split}  \label{eq:intro_SDE} 
dY_{t}^{\eta}(x)&=-\nabla \left(R(Y_{t}^{\eta}(x))+\frac{\eta}{4}|\nabla R(Y_{t}^{\eta}(x))|^{2}\right)dt+\sqrt{\eta}\Sigma(Y_{t}^{\eta}(x))^{1/2}dW_{t}
\end{split}
\end{align}
where $Y_0^\eta(x)=x$ for $x\in\R^d$, $W_t$, $t\geq 0$, is a Brownian motion in $\R^d $ and $\Sigma:\R^d \to \R^{d\times d}$ is the matrix defined by 
\begin{equation} \label{eq:Sigma}
  \Sigma(y)=\E_{\m}\left[ (\nabla_y \tilde R(y,\theta)-\nabla R(y))\otimes (\nabla_y \tilde R(y,\theta)-\nabla R(y)) \right], \quad y \in \R^d.
\end{equation}
Indeed, this incorporates a certain degree of universality of (\ref{eq:intro_SDE}), since the noise in SGD is represented in (\ref{eq:intro_SDE}) by Brownian noise. In machine learning, and, in particular, in overparameterized settings, the covariance matrix $\Sigma$ is typically degenerate. As a result, the square root $\Sigma^{1/2}$ appearing in (\ref{eq:intro_SDE}) has limited regularity properties\footnote{The simple example $\Sigma(y)=y^2$, $\Sigma^{1/2}(y)=|y|$ shows that not more than Lipschitz continuity can be expected from $\Sigma^{1/2}$ in general.}, which makes the analysis of (\ref{eq:intro_SDE}) challenging, and leads to assumptions on $\Sigma^{1/2}$ that are in general not known to hold. The first contribution of this work is to resolve this issue by introducing a new model for the stochastic limiting dynamics, which we name stochastic modified flow (SMF), 
\begin{equation}\label{eq:intro_SMF}
\begin{split}
dX_t^{\eta}(x)&=-\nabla\left(R(X_t^{\eta}(x))+\frac{\eta}{4}|\nabla R(X_t^{\eta}(x))|^2\right)dt+\sqrt{\eta}\int_\Theta G(X_t^{\eta}(x),\theta)W(d\theta,dt),\\
X_0^\eta(x)&=x,\quad x\in\R^d, 
\end{split}
\end{equation}
where $G(x,\theta)=\nabla \tilde R(x,\theta)-\nabla R(x)$ and $W$ is a cylindrical Wiener process on the space $\L{\Theta}{\R}{\m}$.
It is important to notice that (\ref{eq:intro_SMF}) satisfies the same martingale problem as (\ref{eq:intro_SDE}), while avoiding the appearance of $\Sigma^\frac{1}{2}$, thereby bypassing the resulting irregularity of the diffusion coefficients. In contrast, only regularity assumptions on the individual losses $\tilde R$ are needed. More precisely, we get the following result.
\begin{theorem}[see Theorem~\ref{the_main_result} and Corollary~\ref{cor_comparison_for_one_point_motion}]\label{the_one_point_motion}
Let $\tilde R(\cdot,\theta)$ be regular enough for $\m$-a.e. $\theta\in\Theta$ and let $T>0$. Then for every $f\in\Cf^4_b(\R^d)$, one has
\[
\sup_{x\in\R^d}\sup_{n:n\eta\leq T}\left|\E f(X_{n\eta}^\eta(x))-\E f(Z_{n}^\eta(x))\right|\lesssim \eta^2.
\]
\end{theorem}

\textbf{(2) Matching multi-point statistics:} In a variety of works, a dynamical systems approach to the dynamics of SGD has been introduced \cite{Sa.Ts.Fu2022,Wu.Ma.E2018}. This aims at using the concepts of attractors, Lyapunov exponents, stochastic synchronization etc.\ in the analysis of SGD dynamics, for example, in order to analyze asymptotic global stability, that is, if
\begin{equation}
|Z_{n}^{\eta}(x)-Z_{n}^{\eta}(y)|\to0\quad\text{for }n\to\infty\label{eq:intro_synchr}
\end{equation}
in probability. As before, the systematic analysis of such dynamical behavior of SGD relies on the identification of appropriate universal limiting models. It is thus tempting to analyze the dynamical features of SGD by means of those of \eqref{eq:intro_SDE}. However, this is not correct, since \eqref{eq:intro_SDE} only captures the single-point motion of SGD, while dynamical features like stability \eqref{eq:intro_synchr} are properties of the multi-point motions.  More precisely, \eqref{eq:intro_SDE} captures the limiting behavior of the law of single motions $\law(Z_n^{\eta}(x))$, but not the joint multi-point laws $\law(Z_n^\eta(x_1),\ldots,Z_n^\eta(x_m))$ (see also Example~\ref{exa:two_point_motion}). As a second main contribution, in this work we prove that (SMF), in contrast to \eqref{eq:intro_SDE}, captures the correct multi-point distributions of SGD, and therefore opens the way for an analysis of the dynamical properties of its (stochastic) flow. 
\begin{theorem}[see Theorem~\ref{the_main_result} and Corollary~\ref{cor_muly_point_motion}] \label{the_M_point_motion}
Under the assumption of Theorem~\ref{the_one_point_motion}, for every $\Phi\in\Cf^4_b(\cP_2(\R^d))$ one has
\[
\sup_{\mu\in\cP_2(\R^d)}\sup_{n:n\eta\leq T}\left|\E\Phi\left(\mu\circ (X_{n\eta}^\eta)^{-1}\right)-\E\Phi\left(\mu\circ (Z_{n}^\eta)^{-1}\right)\right|\lesssim \eta^2,
\]
where $\mu\circ f^{-1}$ denotes the push forward of the measure $\mu$ under a map $f$. Furthermore, for every $m\in\N$ and $f\in\Cf_b^4(\R^{dm})$,
\[
\sup_{x_1,\ldots,x_m\in\R^d}\sup_{n:n\eta\leq T}\left|\E f(X_{n\eta}^\eta(x_1),\ldots, X_{n\eta}^\eta(x_m))-\E f(Z_{n}^\eta(x_1),\ldots,Z_n^\eta(x_m))\right|\lesssim \eta^2.
\]
\end{theorem}
\textbf{(3) Overparameterized limits:}  As a third main contribution, we extend the small learning rate limit to also incorporate the infinite width limit. We here consider networks with quadratic loss function. Let $\cD \subseteq  \R^{n_0}\times \R^{k_0}$ be a given training data set\footnote{For simplicity we assume that the ground-truth is given by a function $f:\R^{n_0} \to \R^{k_0}$.} with inputs $\Theta =\{ \theta:\ (\theta,f(\theta)) \in \cD \}$ and labels $\{ f(\theta):\ (\theta,f(\theta))\in \cD \}$. For the approximation of $f$ we choose a parameterized hypotheses space $\mathcal M:= \{f^M(z,\cdot): z \in \R^{Md}\}$, $M,d \in \N$, where 
\begin{equation}\label{equ_function_f_M}
  f^M(z,\theta)= \frac{1}{ M }\sum_{ i=1 }^{ M } \Psi(z^i,\theta), \quad \theta \in \Theta,
\end{equation}
with $\Psi:\R^d \times \Theta \to \R^{k_0}$, $z=(z^i)_{i \in [M]}$ and $[M]:=\{1,\ldots,M\}$.
For example, one can choose $\mathcal M$ to be the space of response functions of fully connected feed-forward neural networks with one hidden layer containing $M$ hidden neurons. In that case, we choose a function $\phi:\R \to \R$, the activation function, and we write $z=(z^i)_{i \in [M]}$ with $z^i=(c^i,U^i,b^i) \in \R^{k_0}\times \R^{n_0}\times \R$ and $\Psi(z^i, \theta)=c^i \phi(U^i\cdot \theta+b^i)$. Then,
$$
    f^M(z,\theta)=\frac 1M \sum_{i=1}^M c^i \phi(U^i \cdot \theta+b^i), \quad \theta \in \Theta.
$$
The aim of risk minimization (with respect to the square loss) is to select a suitable model $f^M(z,\cdot )$ minimizing the risk $R(z)=\E_{\m}\tilde R(z,\theta)$, $z \in \R^{Md}$, for
\[
  \tilde R(z, \theta)= \frac{1}{ 2 }   |f(\theta)-f^M(z,\theta)|^2, \quad z \in \R^{Md}, \theta \in \Theta.
\]
As before, this optimization task is executed by the stochastic gradient descent algorithm~\eqref{equ_SGD_non_measure_depended_case} with the starting value $Z_0^\eta=(Z_0^{i,\eta})_{i \in [M]}$ being a tuple of i.i.d. random variables with distribution $\mu \in \cP_2(\R^d)$ that are independent of $\theta_n$, $n\in\N_0$.

A simple computation gives that 
\begin{align*}
  R(z)=  C_f- \frac{1}{ M }\sum_{ i=1 }^{ M } F(z^i)+ \frac{1}{ 2M^2 }\sum_{ i,j=1 }^{ M } K(z^i,z^j),
\end{align*}
where $C_f=\frac 12 \E_{\m}|f(\theta)|^2$ and
\begin{equation}\label{equ_definition_of_F_and_K}
    F(z^i)=\E_{\m}\left[ f(\theta) \cdot \Psi(z^i,\theta) \right],\quad K(z^i,z^j)=\E_{\m}\left[ \Psi(z^i,\theta) \cdot \Psi(z^j,\theta) \right].
\end{equation} 
Taking 
\begin{equation}\label{equ_defintion_of_V_and_G_in_introduction}
\begin{split}
  V(\nu,z^i)&= \nabla F(z^i)- \int_{\R^d} \nabla_{z^i}K(z^i,y ) \nu(dy),\\
  G(\nu,z^i,\theta)&=\left( f(\theta)- \int_{ \R^d  }   \Psi(y,\theta)\nu(dy)  \right)\nabla_{z^i}\Psi(z^i,\theta)\\
  &- \E_{\m}\left[ \left( f(\theta)- \int_{ \R^d  }   \Psi(y,\theta)\nu(dy)  \right)\nabla_{z^i}\Psi(z^i,\theta) \right]
\end{split}
\end{equation}
and replacing $\eta$ in~\eqref{equ_SGD_non_measure_depended_case} by $M\eta$, we can rewrite the expression for the dynamics of $Z_n^\eta=\left( Z^{i,\eta}_n \right)_{i \in [M]}$, $n \in \N_0$, as follows
\begin{equation} 
  \label{equ_sgd_overpar}
  \begin{split}
  Z^{i,\eta}_{n+1}&=Z^{i,\eta}_n+\eta V(\Gamma^{M,\eta}_n,Z^{i,\eta}_n)+\eta G(\Gamma^{M,\eta}_n,Z^{i,\eta}_n,\theta_n),\\
  \Gamma^{M,\eta}_n&= \frac{1}{ M } \sum_{ j=1 }^{ M } \delta_{Z^{j,\eta}_n}, \quad i \in [M],\ \  n \in \N_0,
  \end{split}
\end{equation}
where $\delta_{z}$ denotes the $\delta$-measure in $z$.

We obtain quantified estimates on the approximation of the dynamics of the empirical measure $\Gamma_n^{M,\eta}$, $n \in \N_0$, of SGD by the solution to a distribution dependent stochastic modified flow (DDSMF)
\begin{equation}
  \label{equ_modified_sde_with_interaction_for_sgd}
  \begin{split} 
    dX_t^\eta(x)&= \left[V(\Lambda_t^\eta,X_t^\eta(x))- \frac{\eta}{ 4 }\nabla |V(\Lambda_t^\eta,X_t^\eta(x))|^2- \frac{\eta}{ 4 }\left\langle\D |V(\Lambda_t^\eta,X_t^\eta(x))|^2,\Lambda_t^\eta\right\rangle\right]dt\\
    &+\sqrt{ \eta }\int_{ \Theta }   G(\Lambda_t^\eta,X_t^\eta(x),\theta)W(d \theta,dt), \\
    X_0^\eta(x)&= x, \quad \Lambda_t^\eta=\mu\circ (X_t^\eta)^{-1}, \quad x \in \R^d ,\ \ t\geq 0,
  \end{split}
\end{equation}
where $\D$ denotes the differentiation with respect to the measure dependent argument in the sense of Lions\footnote{For more details see Section~\ref{sub:measure_valued_diffusion}}, $\langle \varphi ,\nu  \rangle$ denote the integration of a function $\varphi: \R^d \to \R$ with respect to a measure $\nu$ and $W$ is a cylindrical Wiener process on $\L{\Theta}{\R}{\m}$. We remark that 
\[ 
\frac{1}{2}\left\langle\D |V(\nu,z)|^2,\nu\right\rangle= V(\nu,z)\left\langle\nabla_x\nabla_zK(z,x),\nu(dx)\right\rangle,
\]
according to the form of $V$ in~\eqref{equ_defintion_of_V_and_G_in_introduction} and properties of Lions derivative.

\begin{theorem}[see Theorem~\ref{the_main_result}, Corollary~\ref{cor_overparam_case} and Remark~\ref{rem_relationship_between_M_and_eta}]\label{the_overparametrized_limit}
Let $\Psi$ be regular enough and $T>0$. Then for every $\Phi\in\Cf^4_b(\cP_2(\R^d))$ and $\mu\in\cP_2(\R^d)$ with a finite $p$ moment with $p>2$, one has
\[
\sup_{n:n\eta\leq T}\left|\E\Phi(\Lambda_{n\eta}^\eta)-\E\Phi(\Gamma_n^{M,\eta})\right|\lesssim \eta^2
\]
for every $\eta>0$ and $M$ large enough.
\end{theorem}
This extends the framework of SMEs and SMFs to (\ref{equ_modified_sde_with_interaction_for_sgd}) which can thus serve as the starting point to analyze the stochastic dynamics of SGD in large, shallow networks.

{\bf Overview of the literature.} Stochastic modified equations as limiting objects of SGD in the regime of small learning rates have been introduced by Li, Tai and E in~\cite{Li:2017, Li_Tai:2019}. Following these
original papers several results were derived for diffusion approximations with SMEs, e.g. generator based proofs~\cite{Feng:2018, Hu:2019} and uniform-in-time estimates for strongly convex objective functions~\cite{Lei:2022}. For a discussion on the validity of the diffusion approximation for finite (non-infinitesimal) learning rate see~\cite{li2021validity}.

In~\cite{Ro.Va2018,Ch.Ba2018,Sirignano_LLN_2020,Javanmard_Mondelli:2020,Mei_Montanari:2018}, the convergence of gradient descent dynamics for overparameterized neural networks to a Wasserstein gradient flow has been analyzed. The conservative SPDE describing the mean-field limit that incorporates the fluctuations of the stochastic gradient descent was suggested by Rotskoff and Vanden-Eijnden in~\cite{Ro.Va2018a,Rotskoff:2022}. The rigorous study of the well-posedness of this conservative SPDE and proof of quantified central limit theorem has been done in~\cite{Gess_SMFE:2022}, using the observation that its solutions can be described by the SDE with interaction~\eqref{equ_sde_with_interaction} below, which was investigated e.g., in~\cite{Dorogovtsev:2007,Dorogovtsev:2010,Pilipenko:2006,Belozerova:2020,Wang_Feng:2021}  (see also~\cite{Dorogovtsev:2002,Wang_Feng:2021,Kurtz:1999,Carmona_Delarue:2016} for its connection with McKean--Vlasov SDEs with common noise). It should be noted that the stochastic modified flows proposed in this work are of a particular form of the SDE with interaction~\eqref{equ_sde_with_interaction}.  In \cite{Sirignano:2020,Ro.Va2018}, a linear SPDE has been rigorously identified in the context of central limit fluctuations of stochastic gradient descent in the overparameterized regime.

{\bf The paper is organized as follows:} In Section~\ref{sec:measure_valued_diffusion_and_sde_with_interaction}, we introduce a stochastic differential equation with interaction (see~\eqref{equ_sde_with_interaction}) that covers both the SMF~\eqref{eq:intro_SMF} and the DDSMF~\eqref{equ_modified_sde_with_interaction_for_sgd} and recall existence and uniqueness results assuming Lipschitz-continuity of its coefficients. Moreover, we state a result for the continuous dependence of solutions to the SDE with interaction with respect to its initial distribution, as well as an analog of Kolmogorov's equation in the setting of SDEs with interaction. Section~\ref{sub:uniform_in_time_diffusion_approximation} is devoted to the main result of this article, Theorem~\ref{the_main_result}, which compares the dynamics of a discrete time Markov chain with those of a solution to a corresponding SDE with interaction. Theorem~\ref{the_one_point_motion}, Theorem~\ref{the_M_point_motion} and Theorem~\ref{the_overparametrized_limit} then follow as consequences of Theorem~\ref{the_main_result}, see Corollary~\ref{cor_comparison_for_one_point_motion}, Corollary~\ref{cor_muly_point_motion} and Corollary~\ref{cor_overparam_case}, respectively.

\section{Measure-valued diffusion and stochastic modified flows}
\label{sec:measure_valued_diffusion_and_sde_with_interaction}

The goal of this section is to prove the well-posedness for stochastic modified flows and investigate some properties of the associated semigroup. We recall that $\cP_2(\R^d )$ denotes the space of probability measures $\mu$ on $\R^d $ such that
\[
  \int_{ \R^d  }   |x|^2\mu(dx)<\infty 
\]
 with the Wasserstein distance defined by
\[
  \cW_2(\mu,\nu)= \inf\limits_{ \chi \in \Pi(\mu,\nu) }\left( \int_{ \R^d  }   \int_{ \R^d  }   |x-y|^2\chi(dx,dy)   \right),
\]
where $\Pi(\mu,\nu)$ is the set of all probability measures on $\R^d \times \R^d $ with marginals $\mu$ and $\nu$. It is well know that $\cP_2(\R^d )$ equipped with the Wasserstein distance $\cW_2$ is a Polish space.

Let $\L{E}{\R^k}{\nu}$ be the space of all $2$-integrable functions from a  measure space $(E,\cE,\nu)$ to $\R^k $ with the usual inner product $\langle \cdot  ,\cdot   \rangle_{\nu}$ and the associated norm $\|\cdot \|_{\nu}$. We will further fix a measure space $(\Theta,\cG,\m)$ such that $\m$ is a finite measure and the space $\L{\Theta}{\R}{\m}$ is separable. We will also consider a cylindrical Wiener process $W_t$, $t\geq 0$, on $\L{\Theta}{\R}{\m}$ defined on a filtered complete probability space $(\Omega,\F, (\F_t)_{t\geq0}, \p)$, that is,
\begin{enumerate}
    \item [(i)] for every $t\geq 0$, the map $W_t:\L{\Theta}{\R}{\m}\to\L{\Omega}{\R}{\p}$ is linear;
    \item[(ii)] for every $h\in\L{\Theta}{\R}{\m}$, $W_t(h)$, $t\geq0$, is an $(\F_t)_{t\geq 0}$-Brownian motion with $\Var W_t(h)=\|h\|_{\m}^2t$.
\end{enumerate}
We will assume that $(\F_{t})_{t\geq 0}$ is the complete right-continuous filtration  generated by $W_t$, $t\geq 0$. For an $(\F_{t})_{t\geq 0}$-progressively measurable $\L{\Theta}{\R^k}{\m}$-valued process $g(t,\cdot)=\{g(t,\theta),\ \theta\in\Theta\}$, $t \geq 0$, with 
\[
  \int_{ 0 }^{ t } \|g(s,\cdot)\|_{\m}^2 ds<\infty 
\]
a.s. for every $t\geq 0$, we will write\footnote{For the definition of the integral with respect to a cylindrical Wiener process see, e.g.,~\cite[Section~2.2.4]{Gawarecki:2011}}
\[
  \int_{ 0 }^{ t } \int_{ \Theta }   g(s,\theta)W(d \theta,ds):=\int_{ 0 }^{ t } \Upsilon(s)dW_s  
\]
for $\Upsilon(s)h=\langle g(s,\cdot) , h \rangle_{\m}=\left(\langle g_i(s, \cdot ) , h \rangle_{\m}\right)_{i \in [k]}$, $h \in \L{\Theta}{\R}{\m}$.

\subsection{Stochastic modified flows}
\label{sub:sde_with_interaction}

For measurable functions $B:[0,\infty)\times \cP_2(\R^d )\times \R^d \to \R^d $, $G:[0,\infty)\times \cP_2(\R^d ) \times \R^d  \to \L{\Theta}{\R^d}{\m}$ and a probability measure $\mu \in \cP_2(\R^d )$, we consider the following stochastic differential equation
\begin{equation}
  \label{equ_sde_with_interaction}
    \begin{split}
      dX_t(x)&= B(t,\Lambda_t,X_t(x))dt+ \int_{ \Theta }   G(t,\Lambda_t, X_t(x), \theta)W(d \theta, dt), \\
       X_0(x)&= x, \quad \Lambda_t=\mu\circ X^{-1}_t, \quad x \in \R^d ,\ \ t\geq 0.
  \end{split}
\end{equation}
It is clear that the equations~\eqref{eq:intro_SMF} and~\eqref{equ_modified_sde_with_interaction_for_sgd} can be written in the form of~\eqref{equ_sde_with_interaction}. Therefore, in this section we will only focus on~\eqref{equ_sde_with_interaction}  which is called the stochastic differential equation with interaction and was studied, e.g. in~\cite{Wang_Feng:2021,Kotelenez:1997,Pilipenko:2006}. Let $\B(E)$ denote the Borel $\sigma$-algebra on a topological space $E$. Following the definition from~\cite[Definition~2.1.1]{Dorogovtsev:2007} or~\cite[Definition 2.5]{Gess_SMFE:2022}, we introduce the notion of a solution to~\eqref{equ_sde_with_interaction}. 

\begin{definition}
  \label{def_definition_of_strong_solution}
  A family of continuous processes $\{ X_t(x),\ t\geq 0 \}$, $x \in \R^d $, is called a {\it (strong) solution} to the SDE with interaction~\eqref{equ_sde_with_interaction} with initial mass distribution $\mu \in \cP_2(\R^d )$ if, for each $t\ge 0$ the restriction of $X$ to the time interval $[0,t]$ is $\B([0,t])\otimes\B(\R^d )\otimes \F_t$-measurable, $\Lambda_t=\mu\circ X^{-1}_t$, $t \geq 0$, is a continuous process in $\cP_2(\R^d )$ and for every $x \in \R^d $, a.s.,
  \[
    X_t(x)=x+\int_{ 0 }^{ t } B(s,\Lambda_s,X_s(x))ds+\int_{ 0 }^{ t } \int_{ \Theta }   G(s,\Lambda_s,X_s(x),\theta) W(d \theta,ds),
  \]
  for all $t\geq 0$. For convenience, we will also call the measure-valued process $\Lambda_t$, $t\geq 0$, a solution to~\eqref{equ_sde_with_interaction}.
\end{definition}

Let $\phi_p(x)=|x|^p$, $x \in \R^d $. The following theorem was proved in~\cite{Gess_SMFE:2022}. See Theorem~2.9 and Corollary~2.10 for the well-posedness and the estimates; the existence of a continuous modification of $X$ was observed in the proof of Theorem~2.9 ibid.

\begin{theorem}
  \label{the_well_posedness_of_sde_with_interaction}
  Assume that the coefficients $B$, $G$ of~\eqref{equ_sde_with_interaction} are Lipschitz continuous with respect to $(\mu,x) \in \cP_2(\R^d )\times\R^d $,  that is, for every $T>0$ there exists $L>0$ such that for each $t \in [0,T]$, $\mu,\nu \in \cP_2(\R^d )$ and $x,y \in \R^d $ 
  \begin{equation}
  \label{equ_lipschitz_continuity}
    \begin{split} 
      |B(t,\mu,x)-B(t,\nu,y)|&+ \|G(t,\mu,x,\cdot )-G(t,\nu,y, \cdot )\|_{\m}\\
      &\leq L\left( \cW_2(\mu,\nu)+|x-y| \right),
    \end{split}
  \end{equation}
  and 
  \begin{equation} 
  \label{equ_boundedness_assumption}
    |B(t,\delta_0,0)|+\| G(t,\delta_0,0,\cdot )\|_{\m}\leq L,
  \end{equation}
  where $\delta_0$ denotes the $\delta$-measure at $0$ on $\R^d $.  Then, for every $\mu \in \cP_2(\R^d )$, there exists a unique strong solution $X_t(x)$, $t\geq 0$, $x \in \R^d$, to the SDE with interaction (\ref{equ_sde_with_interaction}). Moreover, there exists a version of $X_\cdot(x)$, $x \in \R^d $, that is a continuous in $(t,x)$, and for every $T>0$ and $p\geq 2$ there exists a constant $C>0$ such that 
  \[
    \E \sup\limits_{ t \in [0,T] }|X_t(x)|^p\leq C(1+ \langle \phi_p , \mu\rangle+|x|^p),
  \]
  for all $x \in \R^d $. In particular, 
  \[
     \E \sup\limits_{ t \in [0,T] }\langle \phi_p , \Lambda_t \rangle\leq C (1+\langle \phi_p , \mu \rangle),
  \]
  where $\Lambda_t=\mu\circ X_t^{-1}$.
\end{theorem}

From now on, we will only consider the version $X_t(x)$, $t \geq 0$, $x \in \R^d $, of a solution to the SDE with interaction~\eqref{equ_sde_with_interaction} which is continuous in $(x,t)$. In order to reflect the dependency on the initial mass distribution we will write $X_t(\mu,x)$ and $\Lambda_t(\mu)$ instead of $X_t(x)$ and $\Lambda_t$. We next recall the result on the continuous dependence of $\Lambda_t(\mu)$, $t\geq 0$, with respect to the initial condition $\mu$, that was obtained in~\cite[Theorem~2.14]{Gess_SMFE:2022}.

\begin{proposition}\label{pro_continuous_dependence_of_solution}
  Under the assumption of Theorem~\ref{the_well_posedness_of_sde_with_interaction}, for every $T>0$ there exists a constant $C>0$ depending only on $T$ and the Lipschitz constant $L$ such that 
  \[
  \E\sup_{t\in[0,T]}\left|X_t(\mu,x)-X_t(\nu,y)\right|^2\leq C\left(\cW_2^2(\mu,\nu)+|x-y|^2\right)
  \]
  and 
  \[
  \E \sup_{t \in[0,T]}\cW_2^2\left(\Lambda_t(\mu),\Lambda_t(\nu)\right)\leq C\cW_2^2(\mu,\nu)
  \]
  for all $\mu,\nu\in\cP_2(\R^d)$ and $x,y\in\R^d$.
\end{proposition}

\subsection{Measure-valued diffusion}
\label{sub:measure_valued_diffusion}

The goal of this section is to obtain an analog of Kolmogorov's equation for the process $\Lambda_{t}(\mu)$, $t\geq 0$, given in (\ref{equ_sde_with_interaction}). For this purpose, we need to recall the notion of Lions derivative according to~\cite{Cardaliaguet:2013}. We say that a function $f:\cP_2(\R^d ) \to \R^k$ is $L$-differentiable at $\mu$, if there exists an element $\D f(\mu)$ in $\L{\R^d}{\R^k \times \R^{d}}{\mu}$ such that 
\[
  \lim_{ \| h \|_{\mu}\to 0 } \frac{ f(\mu \circ (\id+h)^{-1})-f(\mu)-\langle \D f(\mu) , h \rangle_{\mu} }{ \|h\|_{\mu} }=0,
\]
where $\id$ denotes the identity map on $\R^d $ and the limit is taken over $h \in \L{\R^d}{\R^d}{\mu}$. In this case, $\D f(\mu)$ is called the $L$-derivative of $f$ at $\mu$. We write $f \in \Cf^1(\cP_2(\R^d ))$ if $f$ is $L$-differentiable at every point $\mu \in \cP_2(\R^d )$ and, for every $\mu$, the derivative has a $\mu$-version $\D f(\mu,x)$ such that $\D f(\mu,x)$ is jointly continuous in $(\mu,x) \in \cP_2(\R^d )\times \R^d$. The set of all functions $f \in \Cf^1(\cP_2(\R^d ))$ such that $f(x)$ and $\D f(\mu,x)$ are bounded in $(\mu,x)\in\cP_2(\R^d )\times \R^d $ will be denoted by $\Cf^1_b(\cP_2(\R^d ))$. Moreover, we define $\Cf^{1,1}_b(\cP_2(\R^d )\times \R^m )$ as the set of all functions from $\cP_2(\R^d )\times \R^m $ to $\R^k$ such that $f(\cdot ,x) \in \Cf^1(\cP_2(\R^d ))$, $f(\mu,\cdot) \in \Cf^1(\R^m )$ for all $x \in \R^m $, $\mu \in \cP_2(\R^d )$, and $f(\mu,x)$, $\nabla f(\mu,x)$, $\D f(\mu,x,y)$ are jointly continuous and bounded in $(\mu,x,y) \in \cP_2(\R^d )\times \R^m \times \R^d $. We iteratively define $\Cf^{l,l}_b (\cP_2(\R^d )\times \R^m )$ as a set of all functions $f:\cP_2(\R^d )\times \R^m \to \R^k $ such that $f,\nabla f \in \Cf^{l-1,l-1}_b(\cP_2(\R^d )\times \R^m )$ and $\D f \in \Cf^{l-1,l-1}_b(\cP_2(\R^d)\times \R^{m+d})$. In particular, $f \in \Cf^{2,2}_b(\cP(\R^d )\times \R^m )$ provided the function $f$, and all its derivatives up to the second order, i.e. $\nabla f$, $\D f$, $\nabla^2 f$, $\nabla \D f$, $\D\nabla f$, $\D^2f$, exist, are bounded and jointly continuous.
If $f:\cP_2(\R^d ) \to \R^k $ belongs to $\Cf^{l,l}_b(\cP_2(\R^d )\times \R^m )$, we will write $f \in \Cf^l_b(\cP_2(\R^d ))$.

Similarly to $\Cf^{l,l}_b(\cP_2(\R^d )\times \R^m )$, we define the class $\tilde\Cf^{l,l}_b(\cP_2(\R^d )\times \R^m )$ as the set of continuous and bounded functions $f:\cP_2(\R^d )\times \R^m \to \L{\Theta}{\R^k}{\m}$ such that for $\m$-a.e. $\theta\in\Theta$ we have $f(\cdot ,\cdot ,\theta) \in \Cf^{l,l}_b(\cP_2(\R^d )\times \R^m )$ and all its derivatives up to the $l$-th order are continuous and bounded as $\L{\Theta}{\R^k}{\m}$-valued functions.

\begin{example}
  \label{exa_derivative_of_cylindrical_function}
  If $\varphi_i \in \Cf_b^1(\R^d )$, $i \in [n]$, and $h \in \Cf^1(\R^n )$ then the function $f(\mu)=h(\langle \varphi_1 ,\mu  \rangle,\dots,\langle \varphi_n ,\mu  \rangle)$, $\mu \in \cP_2(\R^d )$, belongs to $\Cf^1_b(\cP_2(\R^d ))$ and 
  \[
    \D f(\mu,x)=\sum_{ i=1 }^{ n } \partial_i h(\langle \varphi_1 ,\mu  \rangle,\dots,\langle \varphi_n ,\mu  \rangle)\nabla \varphi_i(x), \quad x \in \R^d ,\ \ \mu \in \cP_2(\R^d ).
  \]
\end{example}

For the coefficients $B$ and $G$ of the equation~\eqref{equ_sde_with_interaction},  we define the following second-order differential operator 
\begin{equation}
  \label{equ_second_order_differential_operator}
  \begin{split} 
    \cL_tf(\mu):&= \frac{1}{ 2 }\int_{ \R^d  }   \int_{ \R^d  }   \tilde A(t,\mu,x,y):\D^2f(\mu,x,y)\mu(dx)\mu(dy) \\
    &+ \frac{1}{ 2 }\int_{ \R^d  }    A(t,\mu,x):\nabla \D f(\mu,x)\mu(dx)\\
    &+ \int_{ \R^d  }   B(t,\mu,x)\cdot \D f(\mu,x)\mu(dx),
  \end{split}
\end{equation}
for $f \in \Cf^{2}_b(\cP_2(\R^d ))$, where
\begin{align*}
\tilde{A}(t,\mu,x,y)&=\E_{ \m}\left[ G(t,\mu,x,\theta)\otimes G(t,\mu,y,\theta) \right]\\
&=\left( \langle G_i(t,\mu,x,\cdot) , G_j(t,\mu,y,\cdot) \rangle_{\m} \right)_{i,j \in [d]},\\
A(t,\mu,x)&=\tilde{A}(t,\mu,x,x)
\end{align*}
and we use the notation $C:D= \sum_{ i,j=1 }^{ d } c_{i,j}d_{i,j}$
for 
$C=(c_{i,j})_{i,j \in [d]}, D=(d_{i,j})_{i,j \in [d]}
$ and $a \cdot b=\sum_{ i=1 }^{ d } a_ib_i$ for $a=(a_i)_{i \in [d]}$, $b=(b_i)_{i \in [d]}$.

Let us introduce some additional notation. We will write that a function $f:[0,T]\times \cP_2(\R^d )\times \R^m \to \R^k $ belongs to $\Cf^{0,1,1}_b([0,T]\times \cP_2(\R^d )\times \R^m )$ if, for every $t \in [0,T]$, $f(t,\cdot ,\cdot ) \in \Cf^{1,1}_b(\cP_2(\R^d )\times \R^m )$ and the maps $f$, $\nabla f(t,\mu,x)$, $\D f(t,\mu,x,y)$ are jointly continuous and bounded in $(t,\mu,x,y)$. Iteratively, we write $f \in \Cf^{0,l,l}_b([0,T]\times \cP_2(\R^d )\times \R^m )$ provided that $f,\nabla f \in \Cf^{0,l-1,l-1}_b([0,T]\times \cP_2(\R^d )\times \R^m )$ and $\D f \in \Cf^{0,l-1,l-1}_b([0,T]\times \cP_2(\R^d )\times \R^{m+d} )$. The set of all functions $f:[0,T]\times \cP_2(\R^d) \to \R^k $ such that $f \in \Cf^{0,2,2}_b([0,T]\times \cP_2(\R^d )\times \R^m )$ will be denoted by $\Cf^{0,2}_b([0,T]\times \cP_2(\R^d ))$.

Similarly to $\Cf^{0,l,l}_b([0,T]\times \cP_2(\R^d )\times \R^m )$, we also define the class $\tilde\Cf^{0,l,l}_b([0,T]\times \cP_2(\R^d )\times \R^m )$ as the set of continuous and bounded functions $f:[0,T] \times \cP_2(\R^d )\times \R^m \to \L{\Theta}{\R^k}{\m}$ such that for $\m$-a.e. $\theta\in\Theta$ we have $f(\cdot ,\cdot ,\cdot ,\theta) \in \Cf^{0,l,l}_b([0,T] \times \cP_2(\R^d )\times \R^m )$ and all  corresponding derivatives are continuous and bounded as $\L{\Theta}{\R^k}{\m}$-valued functions.

We next provide the well posedness of the Kolmogorov equation associated to~\eqref{equ_sde_with_interaction}. This result can be obtained as in the proof of Theorem~3.1 in~\cite{Wang_Feng:2021} with slight changes, where a similar equation driven by a finite dimensional noise was considered.

\begin{proposition}[Kolmogorov equation]
  \label{pro_kolmogorov_equation}
  Let $T>0$ and the coefficients $B$, $G$ of~\eqref{equ_sde_with_interaction} belong to $\Cf^{0,2,2}_b([0,T]\times \cP_2(\R^d )\times \R^d )$ and $\tilde\Cf^{0,2,2}_b([0,T] \times \cP_2(\R^d )\times \R^d )$, respectively. For $\mu \in \cP_2(\R^d)$, let $\Lambda_t(\mu)$, $t \in [0,T]$, be a solution to \eqref{equ_sde_with_interaction} with initial mass distribution $\Lambda_0(\mu)=\mu$. Then, for every $\Phi \in \Cf_b^{2}(\cP_2(\R^d ))$, the function 
  \[
    U(t,\mu)=\E\Phi(\Lambda_{t}(\mu)), \quad (t,\mu) \in [0,T]\times \cP_2(\R^d ),
  \]
  is a unique solution to the equation 
  \begin{equation} 
  \label{equ_kolmogorov_equation}
    \begin{split} 
      &\partial_t U(t,\mu)=\cL_tU(t,\mu),\\
      &U(0,\mu)=\Phi(\mu), \quad (t,\mu)\in [0,T] \times \cP_2(\R^d ),
    \end{split}
  \end{equation}
  in the class $\Cf_b^{0,2}([0,T]\times \cP_2(\R^d ))$ with $\partial_tU \in \Cf([0,T]\times \cP_2(\R^d ))$.

  If, additionally, $B \in \Cf^{0,l,l}_b([0,T]\times \cP_2(\R^d )\times \R^d )$, $G \in \tilde\Cf^{0,l,l}_b([0,T] \times \cP_2(\R^d )\times \R^d )$ and $\Phi \in \Cf_b^{l}(\cP_2(\R^d ))$, for some $l>2$, then $U \in \Cf_b^{0,l}([0,T] \times \cP_2(\R^d ))$.
\end{proposition}

\section{Diffusion approximation via stochastic modified flows}
\label{sub:uniform_in_time_diffusion_approximation}

The goal of this section is to prove the theorems stated in the introduction. For this, we first show a general result comparing the dynamics of a Markov chain defined below with a corresponding SDE with interaction and cylindrical noise. Then, we show that the results given in the introduction immediately follow from the general comparison statement. We fix measurable functions $V:\cP_2(\R^d )\times \R^d \to \R^d $ and $G: \cP_2(\R^d ) \times \R^d \to \L{\Theta}{\R^d}{\m}$ such that $\E_{\m}G(\mu,x,\theta)=0$ for all $(\mu,x) \in \cP_2(\R^d) \times \R^d$. For $\eta>0$ and $\mu \in \cP_2(\R^d )$, we consider a Markov chain defined by 
\begin{equation} 
  \label{equ_sgd_general_form}
  \begin{split} 
    Z_{n+1}^\eta(z)&= Z_n^\eta(z)+\eta V(\Gamma_n^\eta,Z_n^\eta(z))+\eta G(\Gamma_n^\eta,Z_n^\eta(z),\theta_n),\\
    Z_0^\eta(z)&= z, \quad \Gamma_n^\eta=\mu\circ (Z_n^\eta)^{-1}, \quad z \in \R^d ,\ \ n \in \N_0,
  \end{split}
\end{equation}
where $\theta_n$, $n \in \N_0$, are i.i.d. sampled from the distribution $\m$. We remark that, e.g., the SGD dynamics in the overparameterized shallow neural network in~\eqref{equ_sgd_overpar} can be written in form of~\eqref{equ_sgd_general_form} by taking $\mu= \frac{1}{ M }\sum_{ i=1 }^{ M } \delta_{Z^i_0}$ and $Z^{i,\eta}_n=Z_n^\eta(Z^{i,\eta}_0)$, $i \in [M]$.
We will approximate $\Gamma_n^\eta$, $n\in\N_0$, by solutions to the DDSMF
\begin{equation} 
  \begin{split} \label{equ_modified_sde_general}
    dX_t^\eta(x)&= \left[V(\Lambda_t^\eta,X_t^\eta(x))- \frac{\eta}{ 4 }\nabla |V(\Lambda_t^\eta,X_t^\eta(x))|^2- \frac{\eta}{ 4 }\left\langle\D |V(\Lambda_t^\eta,X_t^\eta(x))|^2,\Lambda_t^\eta\right\rangle\right] dt\\
    &+\sqrt{ \eta }\int_{ \Theta }   G(\Lambda_t^\eta,X_t^\eta(x))W(d \theta,dt), \\
    X_0^\eta(x)&= x, \quad \Lambda_t^\eta=\mu\circ (X_t^\eta)^{-1}, \quad x \in \R^d ,\ \ t\geq 0,
  \end{split}
\end{equation}
where $W$ is a cylindrical Wiener process on $\L{\Theta}{\R}{\m}$. We first prove some auxiliary statements that will imply the well-posedness of the DDSMF~\eqref{equ_modified_sde_general}.

\begin{lemma} 
  \label{lem_property_of_derivative}
  Let $\eps>0$ and $\xi_{s}$, $s \in [0,\eps]$, be a family of square integrable random variables on $\R^k$ defined on a probability space $(\Omega,\F,\p)$. If 
  \[
    \xi_0':=\lim_{ s \to  0+ }\frac{ \xi_s-\xi_0 }{ s }
  \]
  exists in $\L{\Omega}{\R^k}{\p}$, then for every $f \in \Cf^1(\cP_2(\R^k))$ one has 
  \[
    \lim_{ s \to 0+ }\frac{ f(\law (\xi_s))-f(\law (\xi_0)) }{ s }=\E \left[ \D f(\law (\xi_0),\xi_0) \cdot \xi_0' \right].
  \]
\end{lemma}

\begin{proof}
This statement was obtained in~\cite[Lemma~2.4]{Wang_Feng:2021}.
\end{proof}

\begin{lemma}\label{lem_bdd_of_derivatives_implies_lipshitz_continuity}
  Let the functions $V$ and $G$ belong belong to $\Cf_b^{1,1}(\cP_2(\R^d)\times \R^d)$ and $ \tilde \Cf_b^{1,1}(\cP_2(\R^d)\times \R^d)$, respectively. Then, for every $x,y\in \R^d$ and $\mu,\nu\in\cP_2(\R^d)$, we have
  \begin{align*}
      |V(\mu,x)-V(\nu,y)|&+ \|G(\mu,x,\cdot )-G(\nu,y, \cdot )\|_{\m}\\
      &\leq L\left( \cW_2(\mu,\nu)+|x-y| \right),
  \end{align*}
  with
  \begin{align*}
  L&=\sup_{x,y\in\R^d,\mu\in\cP_2(\R^d)}\left( |\nabla V(\mu,x)|+|\D V(\mu,x,y)| \right)\\
  &+\sup_{x,y\in\R^d,\mu\in\cP_2(\R^d)} \left(\|\nabla G(\mu,x,\cdot)\|_{\m}+ \|\D G(\mu,x,y,\cdot)\|_{\m}\right).
  \end{align*}
\end{lemma}

\begin{proof}
Let $x,y\in\R^d$ and $\mu,\nu\in\cP_2(\R^d)$ be fixed. We take an arbitrary probability measure $\chi$ on $\R^d\times\R^d$ with marginals $\mu$, $\nu$ and consider random variables $\zeta_0$, $\zeta_1$ on the probability space $(\R^{d}\times\R^d,\B(\R^{d}\times\R^d),\chi)$ defined by $\zeta_0(x,y)=y$ and $\zeta_1(x,y)=x$ for all $(x,y)\in\R^{d}\times\R^d$. Then $\law(\zeta_0)=\nu$ and $\law(\zeta_1)=\mu$. Set $\xi_s=(1-s)\zeta_0+s\zeta_1$, $s\in[0,1]$, and note that $\xi_i=\zeta_i$, for $i\in\{0,1\}$, and $\xi_s'=(\zeta_1-\zeta_0)$, for all $s\in[0,1]$. We have
\begin{align*}
    |V(\mu,x)-V(\nu,y)|\leq|V(\mu,x)-V(\mu,y)|+|V(\mu,y)-V(\nu,y)|
\end{align*}
and we can bound the terms on the right hand side of the inequality as follows.
With the mean-value theorem, the first term can be  bounded by $\sup_{z \in\R^d,\rho\in\cP_2(\R^d)}|\nabla V(\rho,z)|\,  |x-y|$. To bound the second term, we will use Lemma~\ref{lem_property_of_derivative} and the mean-value theorem:
\begin{align*}
  |V(\mu,y)-&V(\nu,y)|=|V(\law(\zeta_1),y)-V(\law(\zeta_0),y)|\\
    &\leq \sup_{s\in[0,1]}\left| \frac{d}{ds}V(\law(\xi_s),y)\right|\\
      &\leq \sup_{s\in[0,1]}\left| \int_{\R^d}\int_{R^d}\D V(\law(\xi_s),y,\xi_s(z_1,z_2))\cdot\xi_s'(z_1,z_2)\chi(dz_1,dz_2)\right|\\
	&\leq \sup_{z_1,z_2\in\R^d,\rho\in\cP_2(\R^d)}|\D V(\rho,z_1,z_2)|\int_{\R^d}\int_{R^d}|\zeta_1(z_1,z_2)-\zeta_0(z_1,z_2)|\chi(dz_1,dz_2)\\
	&\leq \sup_{z_1,z_2\in\R^d,\rho\in\cP_2(\R^d)}|\D V(\rho,z_1,z_2)|\left(\int_{\R^d}\int_{R^d}|z_2-z_1|^2\chi(dz_1,dz_2)\right)^{ \frac{1}{ 2 }}.
\end{align*}
Taking the infimum over all probability measures $\chi$ on $\R^d\times\R^d$ with marginals $\mu$, $\nu$, we obtain
\[
|V(\mu,y)-V(\nu,y)|\leq \sup_{z_1,z_2\in\R^d,\rho\in\cP_2(\R^d)}|\D V(\rho,z_1,z_2)|\cW_2(\mu,\nu).
\]

The estimate for $\|G(x,\mu)-G(y,\nu)\|_{\m}$ can be obtained similarly.
\end{proof}

Now we are ready to proof the main result of this work.

\begin{theorem} 
  \label{the_main_result}
  Let $V \in \Cf^{5,5}_b(\cP_2(\R^d )\times \R^d )$, $G \in \tilde{\Cf}^{4,4}_b(\cP_2(\R^d )\times \R^d )$ and $\E_{\m}G(\mu,x,\theta)=0$ for all $\mu\in\cP_2(\R^d)$, $x\in\R^d$. For $\mu \in \cP_2(\R^d)$ and $\eta>0$, let $\Gamma_n^\eta(\mu)$, $n \in \N_0$, and $\Lambda_t^\eta(\mu)$, $t\geq 0$, be  defined by \eqref{equ_sgd_general_form}, and \eqref{equ_modified_sde_general}, respectively. Then, for every $\Phi \in \Cf^{4}_b(\cP_2(\R^d ))$ and $T>0$ there exists a constant $C$ independent of $\eta$ such that 
  \begin{equation}\label{equ_the_main_estimate_in_general_case}
    \sup\limits_{ \mu \in \cP_2(\R^d ) }\sup\limits_{ n: n \eta\leq T }\left| \E\Phi(\Lambda_{n\eta}^\eta(\mu))-\E\Phi (\Gamma_n^\eta(\mu)) \right|\leq C\eta^2,
  \end{equation}
  for all $\eta>0$.
\end{theorem}

\begin{proof} 
We first remark that the measure-valued process $\Lambda_t^\eta(\mu)$, $t\geq 0$, is uniquely defined due to Theorem~\ref{the_well_posedness_of_sde_with_interaction} and Lemma~\ref{lem_bdd_of_derivatives_implies_lipshitz_continuity}. Without loss of generality, we consider $\eta\leq T$. The proof of this theorem relies on the comparison of the generators associated with the processes $\Gamma_n^\eta(\mu)$, $n \in \N_0$, and $\Lambda_t^\eta(\mu)$, $t \ge 0$, up to a certain order of $\eta$. We first demonstrate how such a bound on their difference can be used to conclude the proof.

We start from the definition of the transition semigroup for the process $\Gamma_n^\eta(\mu)$, $n \in \N_0$. For convenience of notation, we will drop the superscript $\eta$ in $\Gamma^\eta_n$ and $\Lambda^\eta_t$ and simply write $\Gamma_n$ and $\Lambda_t$, respectively.  Note that $\Gamma_{n+1}=\Gamma_n\circ Y_n^{-1}(\Gamma_n,\cdot)$, where \[
  Y_n(\mu,y)=y+\eta V(\mu,y)+\eta G(\mu,y,\theta_n), \quad \mu \in \cP_2(\R^d ), \ \ y \in \R^d.
\]
Indeed, by~\eqref{equ_sgd_general_form}, $Z_{n+1}(z)=Y_n(\Gamma_n,Z_n(z))$, $z \in \R^d $, and, hence, 
\begin{align*}
  \Gamma_n\circ Y_n^{-1}(\Gamma_n,\cdot )&= \mu\circ Z_n^{-1}\circ Y_n^{-1}(\Gamma_n,\cdot )\\
  &= \mu\circ Y_n(\Gamma_n,Z_n(\cdot ))^{-1} =\mu\circ Z_{n+1}^{-1}=\Gamma_{n+1},
\end{align*}
for all $n \in \N_0$. Therefore, defining the linear operator $\cS$ on the set of all bounded measurable functions $\Psi:\cP_2(\R^d ) \to \R $ by
\[
  \cS \Psi(\mu)=\E_{\m}\Psi(\mu\circ Y_1^{-1}(\mu,\cdot)), \quad \mu \in \cP_2(\mu),
\]
we conclude that 
\begin{equation} 
  \label{equ_formula_for_psi_gamma}
  \begin{split}
    \E_{\m}\Psi (\Gamma_n(\mu))&= \E_{\m}\Psi (\Gamma_{n-1}(\mu)\circ Y^{-1}_{n-1}(\Gamma_{n-1}(\mu),\cdot))\\
    &=\E_{\m}\left[ \E_{\m} \left[ \Psi (\Gamma_{n-1}(\mu)\circ Y^{-1}_{n-1}(\Gamma_{n-1}(\mu),\cdot))\Big|  \Gamma_{n-1}(\mu) \right] \right]\\
    &= \E_{\m}\cS\Psi(\Gamma_{n-1}(\mu))=\dots=\cS^n\Psi(\mu),
  \end{split}
\end{equation}
for all $n \in \N$. Hence, defining $U(t,\mu)=\E\Phi(\Lambda_{t}(\mu))$, $t\geq 0$, $\mu \in \cP_2(\R^d )$, and using~\eqref{equ_formula_for_psi_gamma}, we get for each $\mu \in \cP_2(\R^d )$ and $n \in \N$
  \begin{equation}\label{equ_comparison_of_semigroups}
  \begin{split}
  \E\Phi(\Gamma_n(\mu)))-\E\Phi(\Lambda_{n\eta}(\mu))&=\cS^n\Phi(\mu)-U(t_n,\mu)\\
  &=\sum_{ i=0 }^{ n-1 } \cS^{n-i-1}\left( \cS U(t_{i},\mu)-U(t_{i+1},\mu) \right),
  \end{split}
  \end{equation}
  where $t_i:=i\eta$. 
  
  Thus, by~\eqref{equ_comparison_of_semigroups}, and by the inequality
  \[
    \sup_{\mu \in \cP_2(\R^d )}|\cS\Psi(\mu)|\leq \sup_{ \mu \in \cP_2(\R^d ) }|\Psi(\mu)|,
  \]
 we deduce that there exists a constant $C>0$, such that for all $n \in \N$ with $n\eta \leq T$
  \begin{align}\label{equ_comparison_of_semigroups_2}
  \sup\limits_{\mu \in \cP_2(\R^d )}|\E\Phi(\Lambda_{n\eta}(\mu))-\E\Phi(\Gamma_n(\mu))|&\leq \sup\limits_{ \mu \in \cP_2(\R^d ) }\sum_{ i=0 }^{ n-1 } |\cS U(t_{i},\mu)-U(t_{i+1},\mu)|.
  \end{align}
  In conclusion, to prove~\eqref{equ_the_main_estimate_in_general_case}, it remains to compare $\cS U(t_i,\mu)$ with $U(t_{i+1},\mu)$. For this, we will expand the generators associated with the processes $\Gamma_n^\eta(\mu)$, $n \in \N_0$, and $\Lambda_t^\eta(\mu)$, $t\geq 0$, with respect to $\eta$ up to the second order.

To obtain the expansion of $\cS\Psi(\mu)$ for $\Psi\in\Cf^{3}_b(\cP_2(\R^d))$, we  fix $\mu \in \cP_2(\R^d )$, $\theta \in \Theta$ and consider $Y(\mu,y)=y+\eta V(\mu,y)+\eta G(\mu,y,\theta)$, $y \in \R^d $, as a random variable on the probability space $(\R^d,\B(\R^d ),\mu)$. Define 
  \[
    \xi_s(y)=(1-s)y+ s Y(\mu, y), \quad y \in \R^d ,\ \  s \in [0,1].
  \]
  Then $\xi_0(y)=y$, $\xi_1(y)=Y(\mu,y)$, $\xi'_s(y)=\eta (V(\mu,y)+G(\mu,y,\theta))$  and $\law (\xi_s):=\mu\circ \xi_s^{-1}$ for all $y \in \R^d $, $s \in [0,1]$.   Using Taylor's formula, we obtain
  \begin{equation}\label{equ_expansion_of_Psi}
  \begin{split}
    \Psi(\mu\circ Y^{-1}(\mu,\cdot))&= \Psi(\law(\xi_1))=\Psi(\law(\xi_0))+ \frac{d}{ ds }\Psi(\law(\xi_s))\big|_{s=0}\\
    &+ \frac{1}{ 2 }\frac{ d^2 }{ ds^2 }\Psi(\law(\xi_s))\big|_{s=0}+  \frac{1}{ 2 }\int_{ 0 }^{ 1 } \frac{ d^3 }{ ds^3 }\Psi(\law(\xi_s))(1-s)^3ds.
    \end{split}
  \end{equation}
  We next compute the derivatives appearing in the expression above. By Lemma~\ref{lem_property_of_derivative}, we get 
  \begin{align*}
    \frac{d}{ ds }\Psi(\law(\xi_s))= \eta \int_{ \R^d  }   \D\Psi(\law(\xi_s),\xi_s(x)) \cdot (V(\mu,x)+G(\mu,x,\theta))\mu(dx)
  \end{align*}
  and 
  \begin{align*}
    \frac{ d^2 }{ ds^2 }\Psi(\law(\xi_s))&=  \eta\frac{d}{ ds }\int_{ \R^d  }   \D\Psi(\law(\xi_s),\xi_s(x)) \cdot (V(\mu,x)+G(\mu,x,\theta))\mu(dx)=\\
    &= \eta^2 \int_{ \R^d  }   \int_{ \R^d  }   \D^2\Psi(\law(\xi_s),\xi_s(x),\xi_s(y))\\
    &\qquad\qquad: (V(\mu,x)+G(\mu,x,\theta))\otimes (V(\mu,y)+G(\mu,y,\theta)) \mu(dx)\mu(dy)\\
    &+ \eta^2 \int_{ \R^d  }   \nabla\D\Psi(\law (\xi_s),\xi_s(x))\\
    &\qquad\qquad: (V(\mu,x)+G(\mu,x,\theta))\otimes (V(\mu,x)+G(\mu,x,\theta))\mu(dx).
  \end{align*}
  The third derivative $\frac{ d^3 }{ ds^3 }\Psi(\law(\xi_s))$ can be computed analogously. Since its precise form is not needed, we omit its computation and note only that $\frac{ d^3 }{ ds^3 }\Psi(\law(\xi_s))$, $s\in[0,1]$, is uniformly bounded by $C\|\Psi\|_{\Cf^{3}_b}$ for some constant $C>0$. 
  
  Taking the expectation of~\eqref{equ_expansion_of_Psi} with respect to $\m$ and using the dominated convergence theorem, the equalities $\xi_0(x)=x$, $\law(\xi_0)=\mu$, $ \E_{\m}G(\mu,x,\theta)=0$ and the fact that $\Psi \in \Cf^{3}_b(\cP_2(\R^d ))$, we obtain 
  \begin{equation} 
  \label{equ_expansion_of_psi}
    \begin{split}
      \cS \Psi(\mu)&= \E_{\m}\Psi(\law(\xi_1))=\Psi(\mu)+\eta \int_{ \R^d  }   \D\Psi(\mu,x)\cdot V(\mu,x)\mu(dx)\\
      &+\frac{\eta^2}{2} \int_{ \R^d  }   \int_{ \R^d  }   \D^2\Psi(\mu,x,y): V(\mu,x)\otimes V(\mu,y)\mu(dx)\mu(dy)\\
      &+\frac{\eta^2}{2} \int_{ \R^d  }   \nabla\D\Psi(\mu,x):V(\mu,x)\otimes V(\mu,x)\mu(dx)\\
      &+\frac{\eta^2}{2} \int_{ \R^d  }   \int_{ \R^d  }   \D^2\Psi(\mu,x,y):\tilde A(\mu,x,y)\mu(dx)\mu(dy)\\
      &+\frac{\eta^2}{2} \int_{ \R^d  }   \nabla\D\Psi(\mu,x):A(\mu,x)\mu(dx)+\eta^3R_1(\Psi,\mu) ,
    \end{split}
  \end{equation}
  where $\sup_{\mu \in \cP_2(\R^d )}|R_1(\Psi,\mu)|\leq  C\|\Psi\|_{\Cf^{3}_b}$, for a constant $C>0$ and
  $$
  \tilde{A}(\mu,x,y)=\E_{\m}\left[G(\mu,x,\theta )\otimes G(\mu,y,\theta )\right], \quad  A(\mu,x)=\tilde{A}(\mu,x,x).
  $$

  We next expand the generator of the process  $\Lambda_t^\eta(\mu)$, $t\geq  0$. Recall that $U(t,\mu)=\E\Phi(\Lambda_{t}(\mu))$, $t\geq 0$, $\mu \in \cP_2(\R^d )$. According to Proposition~\ref{pro_kolmogorov_equation}, we can conclude that for every $t\geq t_i$
  \begin{equation} 
  \label{equ_equality_for_u}
    U(t,\mu)=U(t_i,\mu)+ \int_{ t_i }^{ t } \cL U(r,\mu)dr, 
  \end{equation}
  where $\cL=\cL_1+ \eta\cL_2$ and
  \begin{align*}
    \cL_1 U(r,\mu):&=  \int_{\R^d}   V(\mu,x)\cdot \D U(r,\mu,x ) \mu(dx),\\
    \cL_2U(r,\mu):&= \frac{1}{ 2 } \int_{ \R^d  }   \int_{ \R^d  }   \tilde{A}(\mu,x,y):\D^2U(r,\mu,x ,y )  \mu(dx)\mu(dy)\\
    &+ \frac{1}{ 2 }\int_{ \R^d  }   A(\mu,x):\nabla\D U(r,\mu,x )\mu(dx)\\
    &- \frac{1}{ 4 } \int_{ \R^d  }   \nabla|V(\mu,x)|^2\cdot \D U(r,\mu,x) \mu(dx)\\
    &- \frac{1}{ 4 } \int_{ \R^d  }   \int_{ \R^d  }   \D|V(\mu,x)|^2(y)\cdot\D U(r,\mu,x ) \mu(dx)\mu(dy).
  \end{align*}

  Iterating the equality~\eqref{equ_equality_for_u} as in the proof of Lemma~3 in~\cite{Lei:2022}, we obtain
  \begin{equation} 
  \label{equ_expansion_for_ut}
    U(t_{i+1},\mu)=U(t_i,\mu)+\eta\cL_1U(t_i,\mu)+ \eta^2\left( \cL_2+ \frac{1}{ 2 }\cL_1^2 \right) U(t_i,\mu)+\eta^3R_2(\mu),
  \end{equation}
  where $\sup_{\mu \in \cP_2(\R^d )}|R_2(\mu)|\leq C \|U\|_{\Cf^{0,4}_b([0,T]\times \cP_2(\R^d ))}$ for a constant $C>0$. 
  
  In order to compare $\cS U(t_i,\mu)$ and $U(t_{i+1},\mu)$, we next express  $\cL_2+ \frac{1}{ 2 }\cL_1^2$ in terms of the coefficients of the equation~\eqref{equ_modified_sde_general}.  Note that, according to Example~\ref{exa_derivative_of_cylindrical_function}, we have
  \begin{align*}
    \D \cL_1 U(r,\mu,x)&= \nabla\left[ V(\mu,x)\cdot \D U(r,\mu,x)\right]+\int_{ \R^d  }   \D \left[ V(\mu,y)\cdot\D U(r,\mu,y ) \right](x) \mu (dy)\\
    &= \D U(r,\mu,x)\nabla V(\mu,x)+V(\mu,x)\nabla  \D U(r,\mu,x)\\
    &+\int_{ \R^d  }  \D U(r,\mu,y ) \D V(\mu,y,x) \mu(dy)+ \int_{ \R^d  }   V(\mu,y) \D^2U(r,\mu,y ,x)\mu(dy).
  \end{align*}
  Thus, using the equality $ \frac{1}{ 2 }\nabla|V(\mu,x)|^2=V(\mu,x)\nabla V(\mu,x) $ and $ \frac{1}{ 2 }\D |V(\mu,x)|^2(y)= V(\mu,x)\D V(\mu,x,y)$, we get
  \begin{align*}
     \cL_1^2U(r,\mu,x)&=  \frac{1}{ 2 }\int_{ \R^d  }   \nabla |V(\mu,x )|^2 \cdot \D U(r,\mu,x )\mu(dx)\\
    &+\int_{ \R^d  }   \nabla \D U(r,\mu,x) : V(\mu,x)\otimes V(\mu,x)\mu(dx) \\
    &+ \frac{1}{ 2 }\int_{ \R^d  } \int_{ \R^d  }   \D|V(\mu,x)|^2(y) \cdot \D U(r,\mu,x)\mu(dx)\mu(dy) \\
    &+ \int_{ \R^d }   \int_{ \R^d  }   \D^2U(r,\mu,x,y):V(\mu,x)\otimes V(\mu,y)\mu(dx)\mu(dy).
  \end{align*}
  Consequently,  
  \begin{align*}
    \left( \cL_2+ \frac{1}{ 2 }\cL_1^2 \right)U(r,\mu)&= \frac{1}{ 2 } \int_{ \R^d  }   \int_{ \R^d  }   \tilde{A}(\mu,x,y):\D^2U(r,\mu,x ,y )  \mu(dx)\mu(dy)\\
    &+ \frac{1}{ 2 }\int_{ \R^d  }   A(\mu,x):\nabla\D U(r,\mu,x )\mu(dx)\\
    &- \frac{1}{ 4 } \int_{ \R^d  }   \nabla|V(\mu,x)|^2\cdot \D U(r,\mu,x) \mu(dx)\\
    &- \frac{1}{ 4 } \int_{ \R^d  }   \int_{ \R^d  }   \D|V(\mu,x)|^2(y)\cdot\D U(r,\mu,x ) \mu(dx)\mu(dy)\\
    &+\frac{1}{ 4 }\int_{ \R^d  }   \nabla |V(\mu,x )|^2 \cdot \D U(r,\mu,x )\mu(dx)\\
    &+\frac{1}{2}\int_{ \R^d  }   \nabla \D U(r,\mu,x) : V(\mu,x)\otimes V(\mu,x)\mu(dx) \\
    &+ \frac{1}{ 4 }\int_{ \R^d  } \int_{ \R^d  }   \D|V(\mu,x)|^2(y) \cdot \D U(r,\mu,x)\mu(dx)\mu(dy) \\
    &+ \frac{1}{2}\int_{ \R^d }   \int_{ \R^d  }   \D^2U(r,\mu,x,y):V(\mu,x)\otimes V(\mu,y)\mu(dx)\mu(dy)\\
    &= \frac{1}{ 2 } \int_{ \R^d  }   \int_{ \R^d  }   \tilde{A}(\mu,x,y):\D^2U(r,\mu,x ,y )  \mu(dx)\mu(dy)\\
    &+ \frac{1}{ 2 }\int_{ \R^d  }   A(\mu,x):\nabla\D U(r,\mu,x )\mu(dx)\\
    &+ \frac{1}{ 2 }\int_{ \R^d  }   \nabla \D U(r,\mu,x) : V(\mu,x)\otimes V(\mu,x)\mu(dx) \\
    &+ \frac{1}{ 2 }\int_{ \R^d }   \int_{ \R^d  }   \D^2U(r,\mu,x,\mu):V(\mu,x)\otimes V(\mu,y)\mu(dx)\mu(dy).
  \end{align*} 
  
  Comparing~\eqref{equ_expansion_for_ut} with~\eqref{equ_expansion_of_psi} for $\Psi=U(t_i,\cdot)$, we conclude that 
  \begin{align*}
  \cS U(t_{i},\mu)&=U(t_i,\mu)+\eta\cL_1U(t_i,\mu)+ \eta^2\left( \cL_2+ \frac{1}{ 2 }\cL_1^2 \right) U(t_i,\mu)+\eta^3R_1(U(t_{i},\cdot),\mu)\\
  &=U(t_{i+1},\mu)+\eta^3R_1(U(t_{i},\cdot),\mu)-\eta^3R_2(\mu).
  \end{align*}
  
  Inserting into \eqref{equ_comparison_of_semigroups}, and using the fact that $U \in \Cf^{0,4}_b([0,T]\times \cP_2(\R^d ))$ (see Proposition~\eqref{pro_kolmogorov_equation}), yields, for all $n \in \N$ with $n\eta \leq T$
  \begin{align*}
  \sup\limits_{\mu \in \cP_2(\R^d )}|\E\Phi(\Lambda_{n\eta}(\mu))-\E\Phi(\Gamma_n(\mu))|&\leq 
  \sup\limits_{ \mu \in \cP_2(\R^d ) }\sum_{ i=0 }^{ n-1 } \eta^3|R_1(U(t_{i},\cdot ),\mu)-R_2(\mu)|\\
  &\leq C n\eta^3 \le CT\eta^2.
  \end{align*}
  This completes the proof of the theorem.
\end{proof}

\begin{remark} 
  From the proof of Theorem~\ref{the_main_result} one can see that for every $\Phi \in \Cf^{4}_b(\cP_2(\R^d ))$ and $T>0$ there exists a constant $C>0$ such that 
  \[
    \sup\limits_{ \mu \in \cP_2(\R^d ) }\sup\limits_{ n: n \eta\leq T }\left| \E\Phi(\Lambda_{n\eta}(\mu))-\E\Phi (\Gamma_n(\mu)) \right|\leq C\eta,
  \]
  for all $\eta >0$, if $\Lambda_t=\Lambda_t(\mu)$, $t\geq 0$, is defined by the SDE with interaction
  \begin{align*}
    dX_t(x)&= V(\Lambda_t,X_t(x))dt+\sqrt{ \eta }\int_{ \Theta }   G(\Lambda_t,X_t(x))W(d \theta,dt), \\
    X_0(x)&= x, \quad \Lambda_t=\mu\circ X_t^{-1}, \quad x \in \R^d ,\ \ t\geq 0.
  \end{align*}
\end{remark}

We now apply Theorem~\ref{the_main_result} to the comparison of the SGD dynamics and stochastic modified flows considered in the introduction. First, we recover a variant of the statement for stochastic modified equations.

\begin{corollary}\label{cor_comparison_for_one_point_motion}
 Let $Z_{n}^\eta(x)$, $n\in\N_0$, be defined by~\eqref{equ_SGD_non_measure_depended_case} for a loss function $\tilde R$ and $X_t^\eta(x)$, $t\geq 0$, be a solution to~\eqref{eq:intro_SMF}. Let also $\tilde R(\cdot,\theta)\in \Cf^6_b(\R^d)$ for $\m$-a.e. $\theta\in\Theta$ and assume that
\[
\int_\Theta \| \tilde R(\cdot,\theta)\|_{\Cf^6_b}^2\m(d\theta)<\infty.
\]

Then, for every $f\in\Cf^4_b(\R^d)$ and $T>0$, there exists a constant $C>0$ independent of $\eta$ such that
\[
\sup_{x\in\R^d}\sup_{n:n\eta\leq T}\left|\E f(X_{n\eta}^\eta(x))-\E f(Z_{n}^\eta(x))\right|\leq C \eta^2
\]
for all $\eta>0$.
\end{corollary}

\begin{proof}
Using the dominated convergence theorem it is easily seen that the functions $V:=-\nabla R$ and $G:=\nabla \tilde R-\nabla R$ belong to $\Cf_b^{5,5}(\cP_2(\R^d)\times \R^d)$ and $\tilde\Cf^{4,4}_b(\cP_2(\R^d)\times \R^d)$, respectively, where $R=\E_{\m}\tilde R$. Hence, applying Theorem~\ref{the_main_result} to the function $\Phi(\mu)=\langle f,\mu\rangle$, $\mu\in\cP_2(\R^d)$, that trivially belongs to $\Cf^4_b(\cP_2(\R^d))$, we obtain 
\begin{align*}
\sup_{x\in\R^d}\sup_{n:n\eta\leq T}&\left|\E f(X_{n\eta}^\eta(x))-\E f(Z_{n}^\eta(x))\right|\\
&=\sup_{\mu=\delta_x,x\in\R^d}\sup_{n:n\eta\leq T}\left|\E \langle f(X_{n\eta}^\eta),\mu\rangle-\E \langle f(Z_{n}^\eta),\mu\rangle\right|\\
&\leq\sup\limits_{ \mu \in \cP_2(\R^d ) }\sup\limits_{ n: n \eta\leq T }\left| \E\Phi(\Lambda_{n\eta}^\eta(\mu))-\E\Phi (\Gamma_n^\eta(\mu)) \right|\leq C \eta^2,
\end{align*}
for all $\eta>0$ and some constant $C>0$ independent of $\eta$, where $\Lambda_t^\eta(\mu)=\mu\circ (X_t^\eta)^{-1}$ and $\Gamma_n^\eta(\mu)=\mu\circ (Z_n^{\eta})^{-1}$. This completes the proof of the statement.
\end{proof}

\begin{corollary}\label{cor_muly_point_motion}
Under the assumptions of Corollary~\ref{cor_comparison_for_one_point_motion},  for every $m\in\N$, $f\in\Cf_b^4(\R^{dm})$, $\Phi\in\Cf^4_b(\cP_2(\R^d))$ and $T>0$ there exists a constant $C>0$ independent of $\eta$ such that
\begin{equation}\label{equ_m_point_motion}
\sup_{x_1,\ldots,x_m\in\R^d}\sup_{n:n\eta\leq T}\left|\E f(X_{n\eta}^\eta(x_1),\ldots, X_{n\eta}^\eta(x_m))-\E f(Z_{n}^\eta(x_1),\ldots,Z_n^\eta(x_m))\right|\leq C\eta^2
\end{equation}
and
\begin{equation}\label{equ_evolution_of_measure_in_sgd}
\sup_{\mu\in\cP_2(\R^d)}\sup_{n:n\eta\leq T}\left|\E\Phi(\mu\circ (X_{n\eta}^\eta)^{-1})-\E\Phi(\mu\circ (Z_{n}^\eta)^{-1})\right|\leq C\eta^2
\end{equation}
for all $\eta>0$.
\end{corollary}

\begin{proof}
The estimate~\eqref{equ_evolution_of_measure_in_sgd} can be obtained by the same argument as in the proof of Corollary~\ref{cor_comparison_for_one_point_motion}. To prove~\eqref{equ_m_point_motion}, we will apply Corollary~\ref{cor_comparison_for_one_point_motion} to the function
\[
\tilde R^{\ext}(z,\theta)=\tilde R(z_1,\theta)+\ldots +\tilde R(z_m,\theta), \quad z=(z_i)_{i\in[m]}\in\R^{dm},\ \  \theta\in\Theta.
\] 
Note that
\[
    \nabla \tilde R^{\ext}(z,\theta)=\left(\nabla_{z_i}\tilde R(z_i,\theta)\right)_{i\in[m]}
\]
for all $z=(z_i)_{i\in[m]}\in\R^{dm}$ and $\theta\in\Theta$. Defining $Z_n^{\ext,\eta}(x)$, $n\in\N_0$, by~\eqref{equ_SGD_non_measure_depended_case} with $\tilde R$ and $\R^d$ replaced by $\tilde R_{\ext}$ and $\R^{dm}$, respectively, it is easily seen that 
\[
   Z^{\ext,\eta}_n(x)=\left(Z^\eta_n(x_i)\right)_{i\in[m]},\quad n\in\N_0,
\]
for all $x=(x_i)_{i\in[m]}\in\R^{dm}$.

We next set $R^{\ext}(z)=\E_{\m}\tilde R^{\ext}(z,\theta)$, $z=(z_i)_{i\in [m]}$. Then  
\[
   \nabla R^{\ext}(z)=\left(\nabla_{z_i} R(z_i)\right)_{i\in[m]}
\]
and 
\[
  \nabla |\nabla R^{\ext}(z)|^2=\left(\nabla_{z_i}|\nabla_{z_i}R(z_i)|^2\right)_{i\in[m]}
\]
for all $z=(z_i)_{i\in[m]}\in\R^{dm}$. Moreover, 
\[
G^{\ext}(z,\theta):=\nabla\tilde R^{\ext}(z,\theta)-\nabla R^{\ext}(z,\theta)=\left(G(z_i, \theta)\right)_{i\in[m]},
\]
where $G$ is the coefficient of~\eqref{eq:intro_SMF} that equals $\nabla \tilde R-\nabla R$. Under the assumptions of the corollary, equation~\eqref{eq:intro_SMF} with $R$ and $G$ replaced by $R^{\ext}$ and $G^{\ext}$, respectively, has a unique solution $X_t^{\ext,\eta}(x)$, $x\in\R^{dm}$, $t\geq 0$. Moreover,
\[
X_t^{\ext,\eta}(x)=\left(X_t^{\eta}(x_i)\right)_{i\in[m]},\quad t\geq 0,
\]
a.s. for all $x=(x_i)_{i\in[m]}$. Since $\tilde R^{\ext}$ satisfies the assumptions of Corollary~\ref{cor_comparison_for_one_point_motion}, one gets for every $f\in\Cf_b^4(\R^{dm})$
\begin{align*}
    &\sup_{x\in\R^{dm}}\left|\E f(X_{n\eta}^{\ext,\eta}(x))-f(Z_n^{\ext,\eta}(x))\right|\\
    &=\sup_{x_1,\ldots,x_m\in\R^d}\sup_{n:n\eta\leq T}\left|\E f(X_{n\eta}^\eta(x_1),\ldots, X_{n\eta}^\eta(x_m))-\E f(Z_{n}^\eta(x_1),\ldots,Z_n^\eta(x_m))\right|\leq C\eta^2
\end{align*}
for a constant $C>0$ independent of $\eta$. This completes the proof of the statement.
\end{proof}

In the next example, we show that Corollary~\ref{cor_muly_point_motion} cannot hold for the solution to the classical stochastic modified equation \eqref{eq:intro_SDE}, since the distribution of the two-point motion is different from the distribution of the two-point motion of \eqref{eq:intro_SMF}.

\begin{example} \label{exa:two_point_motion}
The covariation of the two-point motion $(X_t^\eta(x),X_t^\eta(\bar x))$, $t\geq 0$, from the SMF~\eqref{eq:intro_SMF} equals
	\begin{equation} \label{eq:111}
	[X^\eta(x), X^\eta(\bar x)]_t = \eta \int_0^t\tilde A(X_s^\eta(x),X_s^\eta(\bar x))ds,\quad t\geq 0,
	\end{equation}
   where $\tilde A(x,y)=\langle G(x,\cdot)\otimes G(y,\cdot)\rangle_{\m}$. However, the covariation of the two-point motion $(Y^\eta_t(x),Y^\eta_t(\bar x))$, $t\geq 0$, obtained from the SDE~\eqref{eq:intro_SDE}, is given by
	\begin{equation}\label{equ_covariance_of_Y}
	[Y^\eta(x), Y^\eta(\bar x)]_t = \eta \int_0^t\Sigma(Y_s(x))^{1/2} \Sigma( Y_s(\bar x))^{1/2} ds,\quad t\geq 0,
	\end{equation}
	for $\Sigma(x)=\tilde A(x,x)$. This implies that the processes $(X^\eta(x),X^\eta(\bar x))$ and $(Y^\eta(x),Y^\eta(\bar x))$ have different distributions in general. We further notice that the covariance of the one step SGD dynamics defined by~\eqref{equ_SGD_non_measure_depended_case} satisfies
$$
	\cov(Z_1^\eta(x),Z_1^\eta(y)) = \eta^2 \tilde A(x,y),
$$
which is comparable with~\eqref{eq:111}, but not with~\eqref{equ_covariance_of_Y}.
\end{example}

Next, we consider the SGD scheme $Z_n^\eta=(Z_n^{i,\eta})_{i \in[M]}$, $n\in\N_0$, incorporating the infinite width limit that is defined by~\eqref{equ_sgd_overpar}, where $Z_0^{i,\eta}$, $i\in[M]$, are i.i.d. random variables sampled from a measure $\mu\in\cP_2(\R^d)$. We prove the convergence of the empirical distribution process $\Gamma_n^{M,\eta}= \frac{1}{M}\sum_{i=1}^M\delta_{Z_n^{i,\eta}}$, $n\in\N_0$, to a mean-field solution $\Lambda_t^\eta=\mu\circ(X_t^\eta)^{-1}$, $t\geq 0$, of the DDSMF defined by~\eqref{equ_modified_sde_with_interaction_for_sgd}.

\begin{corollary}\label{cor_overparam_case}
Let $\mu\in\cP_2(\R^d)$ and $\mu^M=\frac{1}{M}\sum_{j=1}^m\delta_{Z_0^{j,\eta}}$, where $Z_0^{j,\eta}$, $j\in[M]$, are i.i.d. random variables with distribution $\mu$. Let $\Gamma_n^{M,\eta}$, $n\in\N_0$, and $\Lambda_t^\eta$, $t\geq 0$, be as in~\eqref{equ_sgd_overpar} and~\eqref{equ_modified_sde_with_interaction_for_sgd}, respectively, with $\Gamma_0^{M,\eta}=\mu^M$ and $\Lambda_0^\eta=\mu$. Assume that the function $\Psi$ in~\eqref{equ_function_f_M} satisfies: $\Psi(\cdot,\theta)\in \Cf^6_b(\R^d)$ for $\m$-a.e. $\theta\in\Theta$  and 
\[
\int_\Theta \left(\|\Psi(\cdot,\theta)\|_{\Cf^6_b}^2+|f(\theta )|^2\right)  \|\Psi(\cdot,\theta)\|^2_{\Cf^6_b}\m(d\theta)<\infty.
\]
Then, for every $\Phi\in\Cf^4_b(\cP_2(\R^d))$ there exists a constant $C>0$ independent of $\eta$ and $M$ such that
\begin{equation}\label{equ_estimate_in_corollary3}
\sup_{n:n\eta\leq T}\left|\E\Phi(\Lambda_{n\eta}^\eta)-\E\Phi(\Gamma_n^{M,\eta})\right|\leq C\eta^2+C\sqrt{\E\cW_2^2(\mu,\mu^M)}
\end{equation}
for all $\eta>0$ and $M\in\N$. In particular, if $\mu$ has finite $p$th moment for some $p>2$, with $p\not=4$ for $d\leq 4$ and $p\not=\frac{d}{d-2}$ for $d\geq 5$, then for every $a>0$ there exists a constant $C>0$ independent of $\eta$ and $M$ such that
\begin{equation}\label{equ_estimate_in_corol_for_overp}
\sup_{n:n\eta\leq T}\left|\E\Phi(\Lambda_{n\eta}^\eta)-\E\Phi(\Gamma_n^{M,\eta})\right|\leq C\eta^2
\end{equation}
for all $\eta>0$ and $M\in\N$ satisfying $\frac{K(M)}{\eta^4}\leq a$, where
\[
K(M)=\begin{cases}
   M^{-\frac{1}{2}}+M^{-\frac{p-2}{2}}& \mbox{ if } d\leq 3, \\
   M^{-\frac{1}{2}}\ln(1+M)+M^{-\frac{p-2}{2}}& \mbox{ if } d=4,\\
   M^{-\frac{2}{d}}+M^{-\frac{p-2}{2}}& \mbox{ if } d\geq 5.
\end{cases}
\]
\end{corollary}

\begin{proof} 
First, we show that $V \in \Cf^{5,5}_b(\cP_2(\R^d )\times \R^d )$ and $G \in \tilde{\Cf}^{4,4}_b(\cP_2(\R^d )\times \R^d )$, where $V$ and $G$ are given by (\ref{equ_defintion_of_V_and_G_in_introduction}). Analogously to the proof of Corollary~\ref{cor_comparison_for_one_point_motion}, we get that $F \in \Cf_b^6(\R^d)$, where $F(z)=\E_{\m}\left[ f(\theta) \cdot \Psi(z,\theta) \right]$, $z \in \R^d$, and, thus, $\nabla F \in \Cf_b^5(\R^d)$. For $K(z^1,z^2)=\E_{\m}\left[ \Psi(z^1,\theta) \cdot \Psi(z^2,\theta) \right]$, $z_1,z_2 \in \R^d$, we use the dominated convergence theorem to get that $K \in \Cf_b^6(\R^{2d})$.
Using Example~\ref{exa_derivative_of_cylindrical_function}, we get, for $\tilde K(\mu,z)= \langle \nabla_{z}K(z,\cdot ) , \mu \rangle$, $\mu \in \cP_2(\R^d), z \in \R^d$, that
$$
    D \tilde K (\mu,z^1,z^2) = \nabla_{z^2}\nabla_{z^1} K(z^1,z^2),
$$
with analogous expressions for higher derivatives.
Thus, $\tilde K \in \Cf_b^{5,5}(\cP_2(\R^d) \times \R^d)$ and, therefore, $V \in \Cf_b^{5,5}(\cP_2(\R^d )\times \R^d )$. To see that $G \in \tilde{\Cf}_b^{4,4}(\cP_2(\R^d )\times \R^d )$ note that
$$
    \tilde G(\mu,z,\theta) = \left( f(\theta)- \langle \Psi(\cdot,\theta),\mu \rangle  \right)\nabla_{z}\Psi(z,\theta), \quad \mu \in \cP_2(\R^d),\ z \in \R^d,\ \theta \in \Theta,
$$
satisfies $\tilde G \in \tilde \Cf_b^{4,4}(\cP_2(\R^d)\times \R^d)$ and $\E_\m[\tilde G(\cdot, \cdot,\theta)] \in \Cf_b^{4,4}(\cP_2(\R^d) \times \R^d)$.

Note that one needs to check the estimate~\eqref{equ_estimate_in_corollary3} only for $\eta\in(0,T]$. Let $\Lambda_t^\eta(\mu)$, $t\geq 0$, be defined by~\eqref{equ_modified_sde_with_interaction_for_sgd} for every $\mu \in\cP_2(\R^d)$. We next fix $\mu\in\cP_2(\R^d)$ and consider the empirical distribution $\mu^M=\frac 1M\sum_{i=1}^M\delta_{Z^{i,\eta}_0}$ associated with the family of i.i.d. random variables $Z^{i,\eta}_0$, $i\in[M]$, sampled from the distribution $\mu$. By Theorem~\ref{the_main_result}, there exists a constant $C>0$ independent of $\eta$ such that 
\[
\sup_{\mu\in\cP_2(\R^d)}\sup_{n:n\eta\leq T}\left|\E\Phi(\Lambda^\eta_{n\eta}(\mu))-\E\Phi(\Gamma_n^\eta(\mu))\right|\leq C\eta^2
\]
for all $\eta\in(0,T]$, where $\Gamma_n^\eta(\mu)$, $n\in\N_0$, is determined by~\eqref{equ_sgd_general_form} with $V$ and $G$ given by~\eqref{equ_defintion_of_V_and_G_in_introduction}. Therefore, using the equality $\Gamma_n^{M,\eta}=\Gamma_n^{\eta}(\mu^M)$ for all $n \in\N_0$, one has
\begin{align*}
    \sup_{n:n\eta\leq T}&\left|\E\Phi(\Lambda^\eta_{n\eta}(\mu^M))-\E\Phi(\Gamma_n^{M,\eta})\right|=\sup_{n:n\eta\leq T}\left|\E\Phi(\Lambda^\eta_{n\eta}(\mu^M))-\E\Phi(\Gamma_n^\eta(\mu^M))\right|\\
    &=\sup_{n:n\eta\leq T}\left|\E\left[\E\left[\Phi(\Lambda^\eta_{n\eta}(\mu^M)\Big|\cA\right]-\E\left[\Phi(\Gamma_n^\eta(\mu^M))\Big|\cA\right]\right]\right|\\
    &\leq\E\left[\sup_{n:n\eta\leq T}\left|\E\left[\Phi(\Lambda^\eta_{n\eta}(\mu^M)\Big|\cA\right]-\E\left[\Phi(\Gamma_n^\eta(\mu^M))\Big|\cA\right]\right|\right]\\
    &\leq \sup_{\mu\in\cP_2(\R^d)}\sup_{n:n\eta\leq T}\left|\E\Phi(\Lambda^\eta_{n\eta}(\mu))-\E\Phi(\Gamma_n^\eta(\mu))\right|\leq C\eta^2
\end{align*}
for all $\eta\in(0,T]$ and $M\in\N$, where $\cA=\sigma(Z^{i,\eta}_0,\ i\in[M])$.

We next compare $\E\Phi(\Lambda^\eta_{n\eta}(\mu))$ with $\E\Phi(\Lambda^\eta_{n\eta}(\mu^M))$. Applying Lemma~\ref{lem_bdd_of_derivatives_implies_lipshitz_continuity} to $V=\Phi$ and $G=0$, we can estimate 
\[
|\E\Phi(\Lambda^\eta_{n\eta}(\mu))-\E\Phi(\Lambda^\eta_{n\eta}(\mu^M))|^2\leq \|\Phi\|_{\Cf_b^1}^2\E\cW_2^2(\Lambda^\eta_{n\eta}(\mu),\Lambda^\eta_{n\eta}(\mu^M)).
\]
Since the coefficients of the SDE~\eqref{equ_modified_sde_with_interaction_for_sgd} are Lipschitz continuous, where the Lipschitz constant can be chosen independently of $\eta\in(0,T]$ due to the assumptions of the corollary and Lemma~\ref{lem_bdd_of_derivatives_implies_lipshitz_continuity}, we can apply Proposition~\ref{pro_continuous_dependence_of_solution} to bound $\E\cW_2^2(\Lambda^\eta_{n\eta}(\mu),\Lambda^\eta_{n\eta}(\mu^M))$. Thus, there exists a constant $C>0$ independent of $\eta$, $M$ and $n$ such that
\[
\E\cW_2^2(\Lambda^\eta_{n\eta}(\mu),\Lambda^\eta_{n\eta}(\mu^M))\leq C\E\cW_2^2(\mu,\mu^M)
\]
for all $\eta \in(0,T]$, $M\in \N$ and $n\in\N_0$ with $n\eta\leq T$. This completes the proof of the first part of the corollary. 

If $\mu$ has finite $p$th moment for $p>2$ such that $p\not=4$ for $d\leq 4$ and $p\not=\frac{d}{d-2}$ for $d\geq 5$, then, by Theorem~1 in~\cite{Fournier:2015},
\[
\E\cW_2^2(\mu,\mu^M)\leq C_1 \langle \phi_p,\mu\rangle^{\frac{2}{p}}K(M),
\]
where $\phi_p(x)=|x|^p$, $x\in\R^d$, and $C_1>0$ depends only on $p$ and $d$. Assuming that $\frac{K(M)}{\eta^4}\leq a$ for some $a>0$, we get
\begin{align*}
    \sup_{n:n\eta\leq T}\left|\E\Phi(\Lambda_{n\eta}^\eta)-\E\Phi(\Gamma_n^{M,\eta})\right|\leq C\eta^2+C\sqrt{\E\cW_2^2(\mu,\mu^M)}\leq C\eta^2+C\sqrt{aC_1}\langle\phi_p,\mu\rangle^{\frac{1}{p}}\eta^2.
\end{align*}
This completes the proof of the second part of the statement.
\end{proof}

\begin{remark}\label{rem_relationship_between_M_and_eta}
Assume that the measure $\mu\in\cP_2(\R^d)$ has all finite moments in Corollary~\ref{cor_overparam_case}. Then we can choose $p$ so large that the first term in every case of the definition of the constant $K(M)$ dominates. Therefore, the estimate~\eqref{equ_estimate_in_corol_for_overp} holds for all $\eta>0$ and $M\geq \frac{a}{\eta^q}$, where $q=8$ for $d\leq 3$, $q=2d$ for $d\geq 5$ and any $q>8$ for $d=4$, since $\frac{K(M)}{\eta^4}\leq a$ is satisfied for some $a>0$ and large enough $p$. 
\end{remark}

\subsection*{Acknowledgements}
The authors were supported by the Deutsche Forschungsgemeinschaft (DFG, German Research Foundation) – SFB 1283/2 2021 – 317210226.   BG acknowledges support by the Max Planck Society through the Research Group "Stochastic Analysis in the Sciences (SAiS)". The third author thanks the Max Planck Institute for Mathematics in the Sciences for its warm hospitality, where a part of this research was carried out.

\providecommand{\bysame}{\leavevmode\hbox to3em{\hrulefill}\thinspace}
\providecommand{\MR}{\relax\ifhmode\unskip\space\fi MR }
\providecommand{\MRhref}[2]{%
  \href{http://www.ams.org/mathscinet-getitem?mr=#1}{#2}
}
\providecommand{\href}[2]{#2}

\end{document}

%% file: Preambule.tex
\usepackage{amsmath,amsthm,amssymb}
\usepackage{times}
\usepackage{enumerate}
\usepackage[square,sort,comma,numbers]{natbib}
\usepackage[colorlinks,linkcolor=blue,citecolor=blue,urlcolor=blue]{hyperref}
\usepackage{latexsym}
\usepackage{color,graphicx}
\usepackage{float}
\usepackage{color}
\usepackage{comment}
\allowdisplaybreaks

\makeatletter
\def\namedlabel#1#2{\begingroup
    #2%
    \def\@currentlabel{#2}%
    \phantomsection\label{#1}\endgroup
}
\makeatother

\def\titlerunning#1{\gdef\titrun{#1}}
\makeatletter
\def\author#1{\gdef\autrun{\def\and{\unskip, }#1}\gdef\@author{#1}}

\makeatother

\def\subjclass#1{{\renewcommand{\thefootnote}{}%
\footnote{\emph{Mathematics Subject Classification (2010):} #1}}}
\def\keywords#1{\par\medskip
\noindent\textbf{Keywords.} #1}

\newtheorem{theorem}{Theorem}[section]
\newtheorem*{thm*}{Theorem}
\newtheorem{corollary}[theorem]{Corollary}
\newtheorem{lemma}[theorem]{Lemma}
\newtheorem{proposition}[theorem]{Proposition}

\theoremstyle{definition}
\newtheorem{definition}[theorem]{Definition}
\newtheorem{remark}[theorem]{Remark}
\newtheorem{example}[theorem]{Example}

\DeclareMathOperator{\law}{Law}

\DeclareMathOperator{\cov}{cov}
\DeclareMathOperator{\Var}{Var}

\numberwithin{equation}{section}

\newcommand{\eps}{\varepsilon}
\newcommand{\R}{\mathbb{R}}

\newcommand{\N}{\mathbb{N}}
\newcommand{\p}{\mathbb{P}}
\newcommand{\Cf}{\mathrm{C}}
\newcommand{\B}{\mathcal{B}}
\newcommand{\D}{\mathrm{D}}

\newcommand{\F}{\mathcal{F}}

\newcommand{\E}{\mathbb{E}}
\newcommand{\m}{\vartheta}
\newcommand{\cP}{\mathcal{P}}

\newcommand{\cL}{\mathcal{L}}
\newcommand{\cG}{\mathcal{G}}
\newcommand{\cD}{\mathcal{D}}

\newcommand{\cS}{\mathcal{S}}
\newcommand{\cA}{\mathcal{A}}
\newcommand{\cE}{\mathcal{E}}
\newcommand{\cW}{\mathcal{W}}

\newcommand{\id}{\mathrm{id}}
\newcommand{\ext}{\mathrm{ext}}
\renewcommand{\L}[3]{L_2((#1,#3);#2)}

\usepackage[colorinlistoftodos,prependcaption,color=yellow,textsize=tiny]{todonotes}